\documentclass[arxiv=true]{agn_article}

\usepackage[tikzexternalize=false]{agn_all}

\renewcommand{\dd}{\displaystyle}

\newcommand{\singset}{\mathcal{S}}
\newcommand{\WHtilde}{\widetilde{W}_{\mathrm{H}}}

\newcommand{\Wtildegeod}{\Wtilde}%
\newcommand{\varphigeod}{\varphi}%
\newcommand{\singsetgeod}{\singset^{\mathrm{geod}}}

\newcommand{\Wvol}{W_{\textrm{\rm vol}}}

\newcommand{\lambdamax}{\lambda_{\mathrm{max}}}

\renewcommand{\WVL}{W_{\mathrm{VL}}}
\newcommand{\WVLtilde}{\widetilde{W}_{\mathrm{VL}}}

\begin{document}

\title{A polyconvex extension of the logarithmic Hencky strain energy}
\author{%
	Robert J.\ Martin\thanks{%
		Corresponding author: Robert J.\ Martin, \ \ Lehrstuhl f\"{u}r Nichtlineare Analysis und Modellierung, Fakult\"{a}t f\"{u}r Mathematik,
		Universit\"{a}t Duisburg-Essen, Thea-Leymann Str. 9, 45127 Essen, Germany; email: robert.martin@uni-due.de%
	}
	\quad and\quad
	Ionel-Dumitrel Ghiba\thanks{%
		Ionel-Dumitrel Ghiba, \ \ Alexandru Ioan Cuza University of Ia\c si, Department of Mathematics, Blvd.
		Carol I, no.~11, 700506 Ia\c si,
		Romania; Octav Mayer Institute of Mathematics of the
		Romanian Academy, Ia\c si Branch, 700505 Ia\c si and Lehrstuhl f\"{u}r Nichtlineare Analysis und Modellierung, Fakult\"{a}t f\"{u}r Mathematik,
		Universit\"{a}t Duisburg-Essen, Thea-Leymann Str. 9, 45127 Essen, Germany, email: dumitrel.ghiba@uni-due.de, dumitrel.ghiba@uaic.ro%
	}
	\quad and\quad
	Patrizio Neff\thanks{%
		Patrizio Neff, \ \ Head of Lehrstuhl f\"{u}r Nichtlineare Analysis und Modellierung, Fakult\"{a}t f\"{u}r
		Mathematik, Universit\"{a}t Duisburg-Essen, Thea-Leymann Str. 9, 45127 Essen, Germany, email: patrizio.neff@uni-due.de%
	}
}
\maketitle
\begin{abstract}
Adapting a method introduced by Ball, Muite, Schryvers and Tirry, we construct a polyconvex isotropic energy function $W\col\GLpn\to\R$ which is equal to the classical Hencky strain energy
\[
	\WH(F) = \mu\,\norm{\dev_n\log U}^2+\frac{\kappa}{2}\,[\tr(\log U)]^2 = \mu\,\norm{\log U}^2+\frac{\Lambda}{2}\,[\tr(\log U)]^2
\]
in a neighborhood of the identity matrix $\id$; here, $\GLpn$ denotes the set of $n\times n$-matrices with positive determinant, $F\in\GLpn$ denotes the deformation gradient, $U=\sqrt{F^TF}$ is the corresponding stretch tensor, $\log U$ is the principal matrix logarithm of $U$, $\tr$ is the trace operator, $\norm{X}$ is the Frobenius matrix norm and $\dev_n X$ is the deviatoric part of $X\in\R^{n\times n}$. The extension can also be chosen to be coercive, in which case Ball's classical theorems for the existence of energy minimizers under appropriate boundary conditions are immediately applicable. We also generalize the approach to energy functions $\WVL$ in the so-called Valanis-Landel form
\[
	\WVL(F) = \sum_{i=1}^n w(\lambda_i)
\]
with $w\col(0,\infty)\to\R$, where $\lambda_1,\dotsc,\lambda_n$ denote the singular values of $F$.
\end{abstract}

\tableofcontents

\section{Introduction}
\label{section:introduction}

The \emph{quadratic isotropic Hencky energy} $\WH\col\GLpn\to\R$ with%
\footnote{%
	Here and throughout, $\GLpn$ denotes the set of $n\times n$-matrices with positive determinant, $F\in\GLpn$ denotes the deformation gradient, $U=\sqrt{F^TF}$ is the corresponding stretch tensor, $\log U$ is the principal matrix logarithm of $U$, $\tr$ is the trace operator, $\norm{X}$ is the Frobenius matrix norm and $\dev_n X$ is the deviatoric part of $X\in\R^{n\times n}$. The constants $\mu$, $\kappa$ and $\Lambda$ represent the shear modulus, the bulk modulus and the first Lam\'e parameter, respectively.
}
\begin{equation}\label{eq:henckyEnergyIntroduction}
	\WH(F) = \mu\,\norm{\dev_n\log \sqrt{F^TF}}^2+\frac{\kappa}{2}\,[\tr(\log \sqrt{F^TF})]^2 = \mu\,\norm{\log \sqrt{F^TF}}^2+\frac{\Lambda}{2}\,[\tr(\log \sqrt{F^TF})]^2\,,
\end{equation}
which is based on the logarithmic strain measures
\begin{equation}\label{eq:logarithmicStrainMeasures}
	\norm{\dev_n\log U}^2\,,\quad \norm{\log U}^2 \quad\text{ and }\quad [\tr(\log U)]^2=\ln^2(\det U)\,,
\end{equation}
is often used by engineers in geometrically nonlinear elasticity formulations to describe small to moderate elastic strains \cite{Anand79}, notably in applications to metal elasticity. Recently, the Hencky energy and the invariants given in \eqref{eq:logarithmicStrainMeasures} have been given a surprising independent motivation as a \emph{geodesic distance measure} of the deformation gradient to the special orthogonal group $\SOn$ of rotations \cite{agn_neff2015geometry}: If $\dg$, $\dist_{\mathrm{geod},\SLn}$ and $\dist_{\mathrm{geod},\mathbb{R}_+\cdot \id}$ denote the canonical left invariant geodesic distances on the Lie-groups $\GLn$, $\SLn\colonequals\{X\in \GLn\;|\det(X)=1\}$ and $\mathbb{R}_+\cdot\id$, respectively, then \cite{agn_neff2015geometry,agn_martin2014minimal}
\begin{align}
	\norm{\log U}^2 &= \dg^2(F, \SOn)\,, \label{eq:geodesicDistanceFull}\\
	\norm{\dev_n\log U}^2 &= {\rm dist}^2_{{\rm geod,{\rm SL}(n)}}\left( \frac{F}{(\det F)^{1/n}}, \SO(n)\right)\,, \label{eq:geodesicDistanceDev}\\
	[\tr(\log U)]^2&=[\log \det U]^2 = \dist_{\mathrm{geod},\mathbb{R}_+\cdot \id}\left((\det F)^{1/n}\cdot \id,
	\id\right)\,. \label{eq:geodesicDistanceDet}
\end{align}

It has been known for a while that the Hencky energy \eqref{eq:henckyEnergyIntroduction} is not overall rank-one convex \cite{agn_neff2000diss,Bruhns01}. However, rank-one convexity (or Legendre-Hadamard ellipticity) is a necessary requirement for polyconvexity, which, in turn, is essential for the applicability of existence proofs based on the direct methods of the calculus of variations \cite{ball1976convexity}. This shortcoming raises some concern regarding the suitability of the Hencky model in finite element methods, although Bruhns et al.\ \cite{Bruhns01} have explicitly determined a (rather large) ellipticity domain of the Hencky energy (cf.\ \cite{agn_ghiba2015ellipticity}).

\subsection{Energy functions in terms of logarithmic strain measures}

In an attempt to overcome this evident shortcoming of the classical Hencky model, the authors recently introduced the so-called \emph{exponentiated Hencky energy} \cite{agn_neff2015exponentiatedI,agn_montella2015exponentiated,agn_nedjar2017finite}
\begin{equation}\label{eq:expHenckyIntroduction}
	\WeH\col\GLp(n)\to\R\,,\quad W_{_{\rm eH}}(F) = \frac{\mu}{k}\,e^{k\,\norm{\dev_n\log U}^2}+\frac{\kappa}{2\,\widehat{k}}\,e^{\widehat{k}\,[(\log \det U)]^2}\,,
\end{equation}
which is rank-one convex (and, in fact, polyconvex) in the planar case\footnote{The planar case, however, is not representative of the general situation, since for isochoric energy functions (including energies based on the isochoric logarithmic strain measure $\norm{\dev_n\log U}^2$; note that $\dev_n\log U = \log\Big(\frac{U}{(\det U)^{\afrac13}}\Big)$), rank-one convexity already implies polyconvexity \cite{agn_martin2015rank} (cf.\ \cite{agn_ghiba2017SL}).} and provides a close approximation of the classical Hencky formulation for sufficiently small strains. However, we have been unable to find another formulation based on the invariants \eqref{eq:logarithmicStrainMeasures} which respects the isochoric-volumetric split and is rank-one convex in dimension $n\geq3$. This motivated us to consider the question of the \emph{possibility} of such a formulation in detail, and indeed it turned out \cite{agn_martin2017taming} that our approach was doomed to fail from the beginning; in particular, for $n\geq3$, there exists no strictly monotone function $\Psi\col[0,\infty)\to\R$ such that either of the energy functions $W\col\GLpn\to\R$ with
\[
	W(F) = \Psi(\norm{\log U}^2) \qquad\text{or}\qquad W(F) = \Psi(\norm{\dev_n\log U}^2)
\]
is elliptic. Furthermore, if $\Psi$ is additionally twice-differentiable, then there exists no smooth function $\Wvol\col(0,\infty)\to\R$ such that the energy $W\col\GLpn\to\R$ with
\[
	W(F) = \Psi(\norm{\dev_n\log U}^2) + \Wvol(\det F)
\]
is elliptic.

\section{Polyconvex extensions of locally elliptic energies}

Since the search for a nontrivial rank-one convex energy function in terms of the scalar-valued logarithmic strain measures turned out to be in vain, it remains to explore alternative methods of finding an elliptic energy function which approximates (or, better yet, is identical to) the Hencky strain energy in the small-strain range. A common way of transforming a non-elliptic function into a rank-one convex one is the computation of its \emph{rank-one convex hull}; however, this approach is not viable in our case (see Appendix \ref{appendix:rankOneConvexHull}).

Instead, we take a different approach and directly modify the quadratic energy expression outside a certain domain of ellipticity. Of course, all physically reasonable energy expressions reduce to linear elasticity in an infinitesimal neighborhood of the identity $\id$ and thus are elliptic in a (finite) neighborhood of $\id$. The question therefore arises whether it is possible to find a rank-one convex or polyconvex \emph{extension} of a given energy outside such a domain of ellipticity.

Note that, whereas rank-one convexity (or Legendre-Hadamard ellipticity) can be considered a \emph{local} property of an energy function (i.e.\ an energy can be elliptic \emph{at some $F\in\GLpn$}), the notion of polyconvexity is only well-defined in a \emph{global} sense. Kristensen \cite{kristensen2000conditions} even gave an example of a function $f\col\R^{2\times2}\to\R$ which is not polyconvex, but can be extended to a smooth polyconvex function $f_B\col\R^{2\times2}\to\R$ from any ball in $B\subset\R^{2\times2}$.

In this note, we construct a polyconvex (and thus rank-one convex) extension of the quadratic-logarithmic Hencky energy \eqref{eq:henckyEnergyIntroduction} and, more generally, for suitable energy expressions of the Valanis-Landel type. In addition, the extension of the Hencky energy considered here is (unconditionally) coercive, which implies an immediate applicability of the direct methods of the calculus of variations to prove the existence of energy minimizers under appropriate boundary conditions.

Our methods are adapted from an approach by Ball, Muite, Schryvers and Tirry \cite{ball2009unpublished}, who considered another strain measure which can be motivated via a distance function on the general linear group $\GLpn$.

\subsection[The Euclidean distance to $\SOn$]{\boldmath The Euclidean distance to $\SOn$}

The \emph{Euclidean distance} of the deformation gradient $F$ to the group $\SOn$ is given by \cite{Grioli40,agn_neff2014grioli}
\[
	\disteuclid^2(F,\,\SOn) = \norm{U-\id}^2 = \sum_{i=1}^n (\lambda_i-1)^2\,,
\]
where $U=\sqrt{F^TF}$ is the stretch tensor and $\lambda_1,\dotsc,\lambda_n$ are the singular values of $F$. Similar to the case of the geodesic distance (i.e.\ the logarithmic strain measure), the mapping $F\mapsto \norm{U-\id}^2$ is not globally rank-one convex \cite{zhang1999various} (cf.\ \cite{dolzmann2012regularity}).

However, Ball et al.\ showed that (a generalization of) this strain measure has a polyconvex extension from a neighborhood of the identity $\id$ to all of $\GLpn$.

\begin{lemma}[Ball, Muite, Schryvers, Tirry \cite{ball2009unpublished}]
\label{lemma:ballPolyconvexExtension}
	For $\frac12<\alpha\leq1$, let
	\[
		\singset_\alpha \colonequals \{ F\in\GLp(n) \setvert \lambda\geq\alpha \text{ for each singular value $\lambda$ of $F$} \}\,.
	\]
	Then the function $W_\alpha\col\singset_\alpha\to\R$ with
	\[
		W_\alpha(F) = \sum_{i=1}^n (\alpha\.\lambda_i^2 - 2\.\lambda_i + 1)
	\]
	for all $F\in\singset_\alpha$ with singular values $\lambda_1,\dotsc,\lambda_n$ has a polyconvex extension to $\GLp(n)$, which is given by
	\[
		\Wtilde_\alpha\col\GLp(n)\to\R\,,\quad \Wtilde_\alpha(F) = \sum_{i=1}^n \varphi_\alpha(\lambda_i) \;- \frac{1}{2\.\alpha}\,\ln(\det F)\,,
	\]
	where
	\[
		\varphi_\alpha(\lambda)=
		\begin{cases}
			\alpha\.\lambda^2 - 2\.\lambda + 1 + \frac{1}{2\.\alpha}\,\ln(\lambda) &:\; \lambda\geq\frac{1}{2\.\alpha}\,,\\
			1-\frac34\.\alpha + \frac{1}{2\.\alpha}\,\ln\left(\frac{1}{2\.\alpha}\right) &:\; \lambda<\frac{1}{2\.\alpha}\,.
		\end{cases}
	\]
\end{lemma}
\begin{proof}
	If $\lambda_i\geq\frac{1}{2\.\alpha}$ for all singular values $\lambda_1,\dotsc,\lambda_n$ of $F$, then
	\[
		\Wtilde_\alpha(F) = \sum_{i=1}^n (\alpha\.\lambda_i^2 - 2\.\lambda_i + 1) + \frac{1}{2\.\alpha}\,\underbrace{\sum_{i=1}^n \ln(\lambda_i)}_{=\ln(\det F)} \,- \frac{1}{2\.\alpha}\,\ln(\det F) = W_\alpha(F)\,.
	\]
	It remains to show that $\Wtilde_\alpha$ is polyconvex. If $\lambda\geq\frac{1}{2\.\alpha}$, then
	\[
		\varphi_\alpha'(\lambda) = 2\.\alpha\.\lambda - 2 + \frac{1}{2\.\alpha\.\lambda} = \frac{(2\.\alpha\.\lambda-1)^2}{2\.\alpha\.\lambda} \geq 0\,,
		\qquad
		\varphi_\alpha''(\lambda) = \frac{2\.\alpha}{\lambda^2}\,\left(\lambda-\frac{1}{2\.\alpha}\right)\,\left(\lambda+\frac{1}{2\.\alpha}\right) \geq 0\,,
	\]
	thus $\varphi_\alpha$ is convex and nondecreasing. According to a criterion by Ball \cite[Theorem 5.1]{ball1976convexity}, the mapping $F\mapsto\sum_{i=1}^n \varphi_\alpha(\lambda_i)$ is therefore convex. Since the mapping $F\mapsto -\frac{1}{2\.\alpha}\,\ln(\det F)$ is convex in $\det F$ and thus polyconvex, the function $\Wtilde_\alpha$ is polyconvex as well.
\end{proof}
\noindent
Applying Lemma \ref{lemma:ballPolyconvexExtension} with $\alpha=1$, we obtain the following corollary.
\begin{corollary}
	The function
	\[
		W\col\singset_{\afrac12}\to\R\,,\quad \disteuclid^2(F,\,\SO(n)) = \norm{\sqrt{F^TF}-\id}^2 = \sum_{i=1}^n (\lambda_i-1)^2
	\]
	has a polyconvex extension from $\singset_{\afrac12}=\{F\in\GLp(n) \setvert \lambda\geq\frac12 \text{ for each singular value $\lambda$ of $F$}\}$ to $\GLp(n)$.
\end{corollary}

\section{Adaptation to logarithmic strain measures}
The ideas laid out in the previous section can be adapted to show that a similar result holds for the logarithmic strain measure (i.e.\ the canonical geodesic distance to $\SOn$, cf.\ \eqref{eq:geodesicDistanceFull}) as well.

\begin{lemma}
\label{lemma:logUsquaredExtension}
	For $\gamma\leq1$, let
	\[
		\singsetgeod_\gamma \colonequals \{ F\in\GLp(n) \setvert e^{\gamma-1}<\lambda<e^\gamma \text{ for each singular value $\lambda$ of $F$} \}\,.
	\]
	Then the function
	\[
		W\col \singsetgeod_\gamma\to\R\,,\quad W(F) = \dg^2(F,\,\SO(n)) = \norm{\log{\sqrt{F^TF}}}^2 = \sum_{i=1}^n \ln^2(\lambda_i)
	\]
	has a polyconvex extension to $\GLp(n)$, which is given by
	\[
		\Wtildegeod_\gamma\col\GLp(n)\to\R\,,\quad \Wtildegeod_\gamma(F) = \sum_{i=1}^n \varphigeod_\gamma(\lambda_i) \;- (2-2\.\gamma)\,\ln(\det F)\,,
	\]
	where
	\[
		\varphigeod_\gamma(\lambda)=
		\begin{cases}
			-(\gamma-1)^2 &:\; \lambda\leq e^{\gamma-1}\,,\\[.42em]
			\ln^2(\lambda)+(2-2\.\gamma)\.\ln(\lambda) &:\; e^{\gamma-1}<\lambda<e^\gamma\,,\\[.42em]
			-\gamma^2 + 2\.\gamma + \frac{2}{e^\gamma}\,(e^{\lambda-e^\gamma}-1) &:\; e^\gamma\leq\lambda\,.
		\end{cases}
	\]
\end{lemma}
\begin{figure}[h!]
	\begin{centering}
		\vspace*{1.4em}
		\tikzsetnextfilename{phiFunctionGraphs.pdf}
 		\begin{tikzpicture}
	 		\pgfmathdeclarefunction*{phi}{1}{\pgfmathparse{ 
	 			(#1<=exp(\gam-1))*%
	 				(-(\gam-1)^2)%
	 			+(and(exp(\gam-1)<#1,#1<=exp(\gam))*%
		 			(ln(#1)^2+(2-2*(\gam))*ln(#1))%
				+(exp(\gam)<#1)*%
					(-(\gam)^2+2*(\gam)+(2/exp(\gam))*(exp(#1-exp(\gam))-1))%
			}}
			
			\begin{axis}[axis line style = thick, axis x line=middle, axis y line=middle, ymax=3.5, ymin=-1.4, xmin=.35, xmax=4.2, every tick/.style={thick}, height=9.1cm, width=16.1cm]
			
		 	\pgfmathsetmacro\gam{1}
		 		\addplot[domain=0.35:4.2,thick,samples=588, color=red] {phi(x)} node[pos=.49, above left] {$\varphi_1$};
			\pgfmathsetmacro\gam{0.4998}
				\addplot[domain=0.35:3.01,thick,samples=588, color=black] {phi(x)} node[pos=.49, below right] {$\varphi_{\afrac12}$};
			\pgfmathsetmacro\gam{.24997}
		 		\addplot[domain=0.35:2.52,thick,samples=588, color=blue] {phi(x)} node[pos=.56, below right] {$\varphi_{\afrac14}$};
		 	\pgfmathsetmacro\gam{0}
		 		\addplot[domain=0.35:2.31,thick,samples=588, color=brown] {phi(x)} node[pos=.49, above left] {$\varphi_0$};
			 
		 	\end{axis}
	 	\end{tikzpicture}
		
 	\end{centering}
 	\caption{\label{fig:phi} The function $\varphigeod_\gamma$ for different values of $\gamma$.}
 \end{figure}
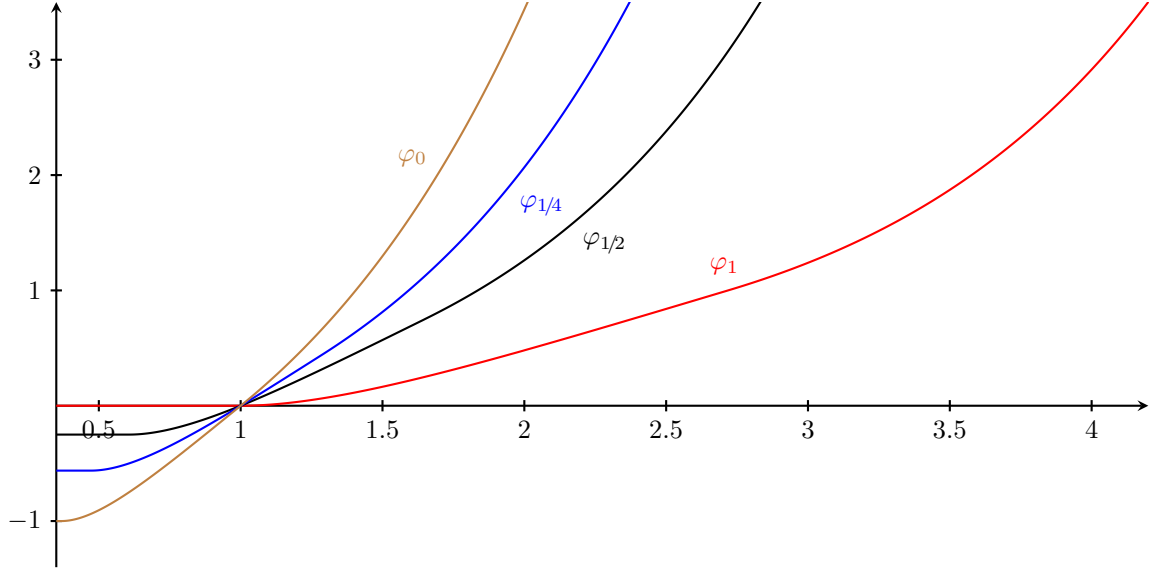
 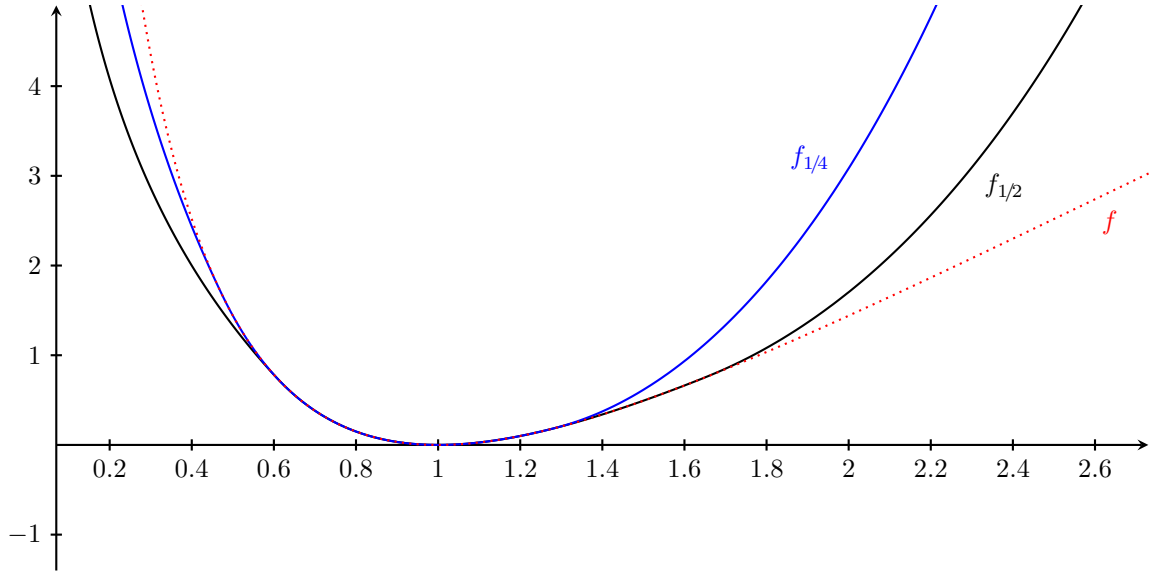
\begin{figure}[h!]
	 \begin{centering}
	 	\vspace*{1.4em}
	 	\tikzsetnextfilename{energyFunctionGraphs.pdf}
	 	\begin{tikzpicture}
		 	\pgfmathdeclarefunction*{phi}{1}{\pgfmathparse{%
			(#1<=exp(\gam-1))*%
					(-(\gam-1)^2)%
				+(and(exp(\gam-1)<#1,#1<=exp(\gam))*%
					(ln(#1)^2+(2-2*(\gam))*ln(#1))%
				+(exp(\gam)<#1)*%
					(-(\gam)^2+2*(\gam)+(2/exp(\gam))*(exp(#1-exp(\gam))-1))%
			}}
			
			\pgfmathdeclarefunction*{phiDetCorrected}{1}{\pgfmathparse{%
				3*(phi(#1)-((2-2*\gam))*ln(#1))%
			}}
			
			\begin{axis}[axis line style = thick, axis x line=middle, axis y line=middle, ymax=4.9, ymin=-1.4, xmin=.07, xmax=2.73, every tick/.style={thick}, height=9.1cm, width=16.1cm]
			
			\pgfmathsetmacro\gam{0.4998}
				\addplot[domain=0.07:2.8,thick,samples=588, color=black] {phiDetCorrected(x)} node[pos=.735, below right] {$f_{\afrac12}$};
			\pgfmathsetmacro\gam{.2499}
				\addplot[domain=0.07:2.8,thick,samples=588, color=blue] {phiDetCorrected(x)} node[pos=.567, above left] {$f_{\afrac14}$};
			\addplot[domain=0.07:2.8,thick,samples=588, color=red, dotted] {3*ln(x)^2} node[pos=.98, below right] {$f$};
			
		\end{axis}
		\end{tikzpicture}
		
	\end{centering}
	\caption{\label{fig:phiDetCorrected} The function $f_\gamma\col\lambda\mapsto \Wtildegeod_\gamma(\lambda\.\id)$ compared to the mapping $f\col\lambda\mapsto \WH(\lambda\.\id)$ with $\mu=1$ and $\Lambda=0$; note the singularity at $\lambda=0$.}
\end{figure}
\begin{corollary}
	In particular (for $\gamma=\frac12$), the function $W$ has a polyconvex extension to $\GLp(n)$ from the set
	\[
		\singsetgeod_{\afrac12} \colonequals \{ F\in\GLp(n) \setvert \tfrac{1}{\sqrt{e}}<\lambda<\sqrt{e} \text{ for each singular value $\lambda$ of $F$} \}\,.
	\]
\end{corollary}
\begin{proof}
	If $e^{\gamma-1}<\lambda<e^\gamma$ for all singular values $\lambda_1,\dotsc,\lambda_n$ of $F$, then
	\[
		\Wtildegeod_\gamma(F) = \sum_{i=1}^n \ln^2(\lambda_i) + (2-2\.\gamma)\,\underbrace{\sum_{i=1}^n \ln(\lambda_i)}_{=\ln(\det F)} \,- (2-2\.\gamma)\.\ln(\det F) = W(F)\,.
	\]
	It remains to show that $\Wtildegeod_\gamma$ is polyconvex. If $e^{\gamma-1}<\lambda<e^\gamma$, then
	\begin{align*}
		\varphigeod_\gamma'(\lambda) &= \frac{2\.\ln(\lambda)}{\lambda} + \frac{2-2\.\gamma}{\lambda} = \frac{2\.\ln(\lambda)+2-2\.\gamma}{\lambda} \geq \frac{2\.\ln(e^{\gamma-1})+2-2\.\gamma}{\lambda} = 0\,,\\[.7em]
		\varphigeod_\gamma''(\lambda) &= \frac{2}{\lambda^2} - \frac{2\.\ln(\lambda)}{\lambda^2} - \frac{2-2\.\gamma}{\lambda^2} = \frac{2\,\gamma-2\.\ln(\lambda)}{\lambda^2} \geq \frac{2\,\gamma-2\.\ln(e^\gamma)}{\lambda^2} = 0\,.
	\end{align*}
	It is easy to see that $\varphigeod_\gamma$ is continuous on $(0,\infty)$ as well as differentiable on $(0,\infty)\setminus\{e^{\gamma-1}\}$ and that $\varphigeod_\gamma'$ is nonnegative and nondecreasing. Thus $\varphigeod_\gamma$ is nondecreasing and convex. Due to Ball's criterion \cite[Theorem 5.1]{ball1976convexity}, the mapping $F\mapsto\sum_{i=1}^n \varphigeod_\gamma(\lambda_i)$ is therefore convex. Since the mapping $F\mapsto -(2-2\.\gamma)\,\ln(\det F)$ is polyconvex for $\gamma\leq1$, the function $\Wtildegeod_\gamma$ is polyconvex as well.
\end{proof}
\noindent Lemma \ref{lemma:logUsquaredExtension} can be applied directly to the classical Hencky strain energy.

\begin{proposition}
\label{prop:henckyExtension}
	Let $\WH$ denote the \emph{quadratic Hencky energy}, given by
	\[
		\WH(F) = \mu\,\norm{\dev_n\log \sqrt{F^TF}}^2+\frac{\kappa}{2}\,[\tr(\log \sqrt{F^TF})]^2 = \mu\,\norm{\log \sqrt{F^TF}}^2+\frac{\Lambda}{2}\,[\tr(\log \sqrt{F^TF})]^2\,,
	\]
	where $\mu$ is the shear modulus, $\kappa$ is the bulk modulus and $\Lambda$ is the first Lam\'e parameter. If $\Lambda\geq0$, then the restriction of $\WH$ to the set
	\[
		\singsetgeod_{\afrac13} = \{ F\in\GLp(n) \setvert e^{-\afrac23}<\lambda<e^{\afrac13} \text{ for each singular value $\lambda$ of $F$} \}
	\]
	has a polyconvex extension to $\GLp(n)$.
\end{proposition}
\begin{proof}
	As in Lemma \ref{lemma:logUsquaredExtension}, let $\Wtildegeod_{\afrac13}$ denote the polyconvex extension of the mapping $F\mapsto W(F)=\norm{\log\sqrt{F^TF}}^2$ from $\singsetgeod_{\afrac13}$ to $\GLp(n)$, and let
	\[
		\WHtilde(F) = \Wtildegeod_{\afrac13}(F) + \psi(\det F)\,,
	\]
	where
	\[
		\psi(t) = \begin{cases}
			\frac{\Lambda}{2}\,\ln^2(t) &:\; t\leq e\,,\\
			\frac\Lambda2 + \frac{\Lambda}{e}\,(e^{t-e}-1) &:\; t>e
		\end{cases}
	\]
	for all $F\in\singsetgeod_{\afrac13}$ with singular values $\lambda_1,\dotsc,\lambda_n$. Then for all $F\in\singsetgeod_{\afrac13}$,
	\[
		\WHtilde(F) = \Wtildegeod_{\afrac13}(F) + \psi(\det F) = W(F) + \psi(\underbrace{\lambda_1\.\dotsc\.\lambda_n}_{\leq e}) = \norm{\log\sqrt{F^TF}}^2 + \frac{\Lambda}{2}\,\ln^2(t) = \WH(F)\,.
	\]
	Furthermore, it is easy to see that the mapping $\psi$ is continuously differentiable with non-decreasing derivative on $(0,\infty)$ and therefore convex. Thus $\WHtilde$ is polyconvex on $\GLp(n)$ as the sum of the polyconvex mapping $\Wtildegeod_{\afrac13}$ and a convex function of $\det F$.
\end{proof}

\begin{figure}[h!]
	\begin{centering}
		\vspace*{1.4em}
		\tikzsetnextfilename{psiVolPartFunctionGraphs.pdf}
		\begin{tikzpicture}
	 		\pgfmathdeclarefunction*{psi}{1}{\pgfmathparse{ 
	 			(#1<=e)*%
	 				(\Lam/2*ln(#1)^2)%
	 			+(exp(1)<#1)*%
		 			(\Lam/2+\Lam/e*(exp(#1-e)-1))%
			}}
			
			\begin{axis}[axis line style = thick, axis x line=middle, axis y line=middle, ymax=7, ymin=-.21, xmin=0, xmax=14.7, ytick={-14}, xtick={1,3,5,7,9,11,13}, every tick/.style={thick}, height=9.1cm, width=16.1cm]
				\addplot[domain=.98:14.7,thick,samples=588, color=red, dotted] {ln(x)^2} node[pos=.84, below right] {$t\mapsto\ln^2(t)$};
				\pgfmathsetmacro\Lam{2}
					\addplot[domain=0.014:7,thick,samples=588, color=black] {psi(x)} node[pos=.35, below right] {$\psi$};
		 	\end{axis}
	 	\end{tikzpicture}
	
	\end{centering}
	\caption{\label{fig:psi} The volumetric part $\psi$ of the polyconvex extension and the original volumetric term $\ln^2(t)$ of the classical Hencky energy.}
\end{figure}
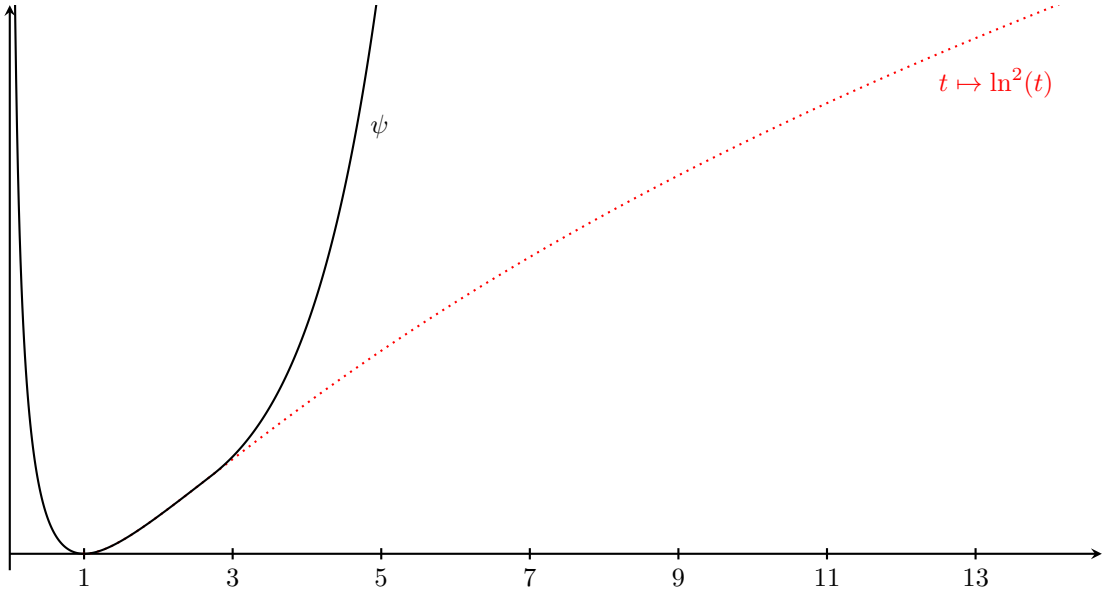

\begin{remark}
	Bruhns et al.\ \cite{Bruhns01} have shown that for $\Lambda\geq0$, the quadratic Hencky strain energy is elliptic on the set of all $F\in\GLp(3)$ with all singular values in the interval $[\alpha,\sqrt[3]{e}]$, where $\alpha\approx0.21<e^{-2/3}$, cf.\ \cite{agn_ghiba2015ellipticity}.
\end{remark}

\begin{remark}
	Note that, even though $\varphi_\gamma(\lambda)$ remains bounded for $\lambda\to0$, the energy $\WHtilde$ exhibits (physically reasonable) singular behavior for $\det F\to0$.
\end{remark}

\subsection{Existence of minimizers}

In order to apply some well-known existence theorems from the direct methods of the calculus of variations to the energy $\WHtilde$, it remains to show that $\WHtilde$ is coercive in the appropriate Sobolev spaces. However, due to the exponential nature of the chosen extension, the following lemma immediately follows from the observation that $\lambdamax>e^{\afrac13}$ for the largest singular value $\lambdamax$ of $F$ if $\norm{F}$ is sufficiently large, in which case
\[
	\WHtilde(F) \geq \Wtildegeod_{\afrac13}(F) \geq -\frac89 + \varphigeod_{\afrac13}(\lambdamax) - \frac43\.\ln(\lambdamax^3) = -\frac13 + \frac{2}{e^{\afrac13}}\.(e^{\lambdamax-e^{\afrac13}}-1) - 4\.\ln(\lambdamax) \geq K_1\.e^{K_3\.\norm{F}} - K_2
\]
for appropriate constants $K_1,K_2,K_3>0$.

\begin{lemma}
	The polyconvex extension $\WHtilde$ of the Hencky energy from Proposition \ref{prop:henckyExtension} is \emph{unconditionally coercive}, i.e.\ for any $p>0$, there exist constants $K_1,K_2>0$ such that
	\begin{align*}
		\WHtilde(F) \geq K_1\.\norm{F}^p - K_2\,.
	\end{align*}
	In particular, for each bounded, connected, open set $\Omega\subset\R^n$ with Lipschitz boundary and any $p\geq1$, the energy functional
	\[
		I\col W^{1,p}(\Omega;\R^n)\to\R\,,\quad I(\varphi) = \int_\Omega \WHtilde(\grad\varphi(x))\,\dx
	\]
	is coercive in the Sobolev space $W^{1,p}(\Omega)$.
\end{lemma}

Since $\WHtilde$ is polyconvex, coercive and bounded below, we can directly apply Ball's classical results on the existence of minimizers for polyconvex energy functions \cite{ball1976convexity} to the energy functional given by $\WHtilde$.

\begin{proposition}
	Let $\Omega\subset \mathbb{R}^n$ be a bounded smooth domain, $\Gamma_D$ be a non-empty and relatively open part of the boundary $\partial\Omega$ and $\varphi_0\in W^{1,q}(\Omega)$ for some $q>1$ such that $\int_\Omega \WHtilde(\grad\varphi_0(x))\,\dx<\infty$. Then there exists at least one $\widehat{\varphi}\in W^{1,p}(\Omega)$ with $\widehat{\varphi}|_{\Gamma_D} = \varphi_0$ such that
	\begin{equation}\label{eq:minimizerExistence}
		\int_\Omega \WHtilde(\grad\widehat{\varphi}(x))\,\dx \;=\; \min\;\left\{\, \int_\Omega \WHtilde(\grad\varphi(x))\,\dx \,\setvert\, \varphi\in W^{1,p}(\Omega)\,,\; \varphi|_{\Gamma_D} = \varphi_0 \,\right\}\,.
	\end{equation}
\end{proposition}

\section{Energy functions in Valanis-Landel form}

We can apply the same extension method to the more general case of \emph{Valanis-Landel type} energy functions, i.e.\ to functions of the form
\[
	\WVL\col\GLpn\to\R\,,\quad \WVL(F) = \sum_{i=1}^n w(\lambda_i)
\]
with a scalar function $w\col(0,\infty)\to\R$. Functions of this type were suggested by Valanis and Landel \cite{valanis1967} as a general hyperelastic model for \emph{incompressible} materials, but are often coupled additively with volumetric energy terms in order to obtain elastic models for compressible materials (including the quadratic Hencky energy $\WH$ as well as Ogden's classical material model \cite{ogden1972large}). Note that the energy $\WVL$ can only be compatible with linear elasticity at the identity $\id$ if $w(1)=0$, $w'(1)=0$ and $w''(1)>0$; the latter two conditions represent the requirements of a stress-free reference configuration and ellipticity at $\id$, respectively.

\begin{proposition}
\label{prop:valanisLandelGeneral}
	Let $w\in C^2((0,\infty))$ such that $w'(1)=0$ and $w''(1)>0$. Then the function
	\[
		\WVL\col\GLpn\to\R\,,\quad \WVL(F) = \sum_{i=1}^n w(\lambda_i)
	\]
	has a polyconvex extension from a neighborhood of the identity $F=\id$ to $\GLpn$.
\end{proposition}
\begin{proof}
	Choose $0<\eps<\frac12$ such that $w''(\lambda)>\frac{w''(1)}{2}$ and $w'(\lambda)>-\frac{w''(1)}{12}$ for all $\lambda\in[1-\eps,1+\eps]$. Let
	\[
		\WVLtilde(F) = \sum_{i=1}^n \varphi(\lambda_i) \; - \frac{w''(1)}{8}\.\ln(\det F)
	\]
	for all $F\in\GLpn$ with singular values $\lambda_1,\dotsc,\lambda_n$, where
	\[
		\varphi(\lambda) =
		\begin{cases}
			w(1-\eps) + \frac{w''(1)}{8}\.\ln(1-\eps) &: \lambda\leq1-\eps\,,\\
			w(\lambda) + \frac{w''(1)}{8}\.\ln(\lambda) &: 1-\eps<\lambda<1+\eps\,,\\
			w(1+\eps) + \left( w'(1+\eps)+\frac{w''(1)}{8\.(1+\eps)} \right) \cdot (\lambda-(1+\eps)) &: 1\leq\lambda\,.
		\end{cases}\,.
	\]
	Then $\varphi$ is continuous and differentiable on $(0,\infty)\setminus\{1-\eps\}$. Furthermore, $\varphi'$ is non-decreasing and non-negative, since
	\begin{alignat*}{3}
		\varphi'(\lambda) &= w'(\lambda) + \frac{w''(1)}{8\.\lambda} \;&&\overset{\lambda<1+\eps<\frac32}{\geq}\; &&w'(\lambda)+\frac{w''(1)}{12} > 0\,,\\
		\varphi''(\lambda) &= w''(\lambda) - \frac{w''(1)}{8\.\lambda^2} \;&&\overset{\lambda>1-\eps>\frac12}{\geq}\; &&w''(\lambda)-\frac{w''(1)}{2} > 0
	\end{alignat*}
	for all $\lambda\in(1-\eps,1+\eps)$. Thus $\varphi$ is convex and non-decreasing, which implies the convexity of the mapping $F\mapsto\sum_{i=1}^n \varphi(\lambda)$ and thus the polyconvexity of $\WVLtilde$. Finally, $\WVLtilde(F)=\WVL(F)$ for all $F\in\GLpn$ with singular values $\lambda_1,\dotsc,\lambda_n$ such that $1-\eps<\lambda_i<1+\eps$ for all $i=1,\dotsc n$.
\end{proof}

\begin{remark}
	Of course, Proposition \ref{prop:valanisLandelGeneral} is also applicable to energy functions of the generalized Valanis-Landel form
	\[
		W(F) = \sum_{i=1}^n w(\lambda_i) + \Wvol(\det F)
	\]
	if $\Wvol\col(0,\infty)\to\R$ is convex in a neighborhood of $1$ (in which case it can easily be extended to a convex function on $(0,\infty)$).
\end{remark}

\section*{Acknowledgment}
The work of Ionel-Dumitrel Ghiba was supported by a grant of the Romanian National Authority for Scientific Research and Innovation, CNCS-UEFISCDI, project number PN-III-P1-1.1-TE-2016-2314.

We thank Sir John Ball (Oxford Centre for Nonlinear PDE) for sharing with us his unpublished work on the polyconvex extension of the Euclidean distance to $\SO(3)$.

\begin{footnotesize}
	\printbibliography[heading=bibnumbered]
\end{footnotesize}
%

\appendix

\section{The rank-one convex hull}
\label{appendix:rankOneConvexHull}

Another approach to finding a rank-one convex \enquote{approximation} of a non-elliptic energy function $W\col\GLpn\to\R$ is the computation of its \emph{rank-one convex hull} $\Wbar$, i.e.\ the largest rank-one convex function below $W$. However, only in few cases can $\Wbar$ be determined analytically. One of these cases is the mapping $F\mapsto\dist^2(F,\SO(2))$, the rank-one convex hull of which is given by \cite{dolzmann2012regularity}
\[
	F\mapsto
	\begin{cases}
		\dist^2(F,\SO(2)) &: \norm{F^+}\geq\frac{\sqrt{2}}{2}\,,\\
		1-2\.\det F &: \text{otherwise}\,,
	\end{cases}\,.
\]
where
\[
	F^+ = \frac12\,\matr{F_{11}+F_{22}&F_{12}-F_{21}\\ F_{21}-F_{12}&F_{11}+F_{22}}
\]
denotes the \emph{conformal part} of $F$.

A fundamental problem of this approach, however, is that the energy function might be changed at points within its domain of ellipticity as well; in particular, the rank-one convex hull $\Wbar$ of $W$ is not necessarily equal to $W$ in a neighborhood of $\id$ and does therefore not induce the same material behavior even for very small strains. For example, the (rank-one) convex hull of the one-dimensional standard double-well potential $x\mapsto(x^2-1)^2$ is given by
\[
	x\mapsto
	\begin{cases}
		0 &: x\in(-1,1),\\
		(x^2-1)^2 &: x\notin(-1,1)\,,
	\end{cases}
\]
thus the energy is changed on every interval $1-\eps,1+\eps$.

\begin{figure}[h!]
	\begin{centering}
		\vspace*{1.4em}
		\tikzsetnextfilename{convexHull.pdf}
 		\begin{tikzpicture}
	 		\pgfmathdeclarefunction*{Wbar}{1}{\pgfmathparse{ 
	 			or(#1<=-1,1<=#1)*%
	 				(#1^2-1)^2%
			}}
			
			\begin{axis}[axis line style = thick, axis x line=bottom, axis y line=left, xtick={-1,0,1}, ytick={0,0.5,1},%
				ymax=1.05, ymin=-.07, xmin=-1.47, xmax=1.47, every tick/.style={thick}, height=4.9cm, width=9.1cm]
			
			\addplot[domain=-2:2,thick,samples=588, color=red] {Wbar(x)} node[pos=.525, above right] {$\Wbar$};
			\addplot[domain=-2:2,very thick,samples=588, color=black, dashed] {(x^2-1)^2} node[pos=.525, above right] {$W$};
			 
		 	\end{axis}
	 	\end{tikzpicture}
		
 	\end{centering}
 	\caption{\label{fig:convexHull} The (rank-one) covex hull $\Wbar$ of the function $W\col\R\to\R\,,\; W(x)=(x^2-1)^2$.}
 \end{figure}

\end{document}